\documentclass{article}
\title{Tails of weakly dependent random vectors
}
\author{Peter Tankov\footnote{Laboratoire de Probabilit\'es et Mod\`eles Al\'eatoires, Universit\'e Paris Diderot, Paris, France. Email:
  \texttt{tankov@math.univ-paris-diderot.fr}}}
\date{}
\usepackage{amsfonts,amsmath,amsthm,amssymb}
\usepackage{graphicx,eufrak}
\usepackage{natbib}
\usepackage{color}
\newtheorem{theorem}{Theorem}
\newtheorem{corollary}{Corollary}
\newtheorem{lemma}{Lemma}
\newtheorem{proposition}{Proposition}
\theoremstyle{remark}

\newtheorem{remark}{Remark}
\theoremstyle{definition}
\newtheorem{definition}{Definition}
\newcommand{\cm}{{\EuFrak B}}
\newcommand{\bs}{\boldsymbol}
\begin{document}
\maketitle
\begin{abstract}
We introduce a new functional measure of tail dependence for weakly dependent
(asymptotically independent) random vectors, termed \emph{weak
  tail dependence function}. The new measure is defined at the level
of copulas and we compute it for several copula families such as
the Gaussian copula, copulas of a class of Gaussian mixture models,
certain Archimedean copulas and extreme value copulas. The new measure allows to quantify the
tail behavior of certain functionals of weakly dependent random vectors at
the log scale. 

\end{abstract}



\section{Introduction}
The goal of this paper is to propose a new approach to describe
asymptotic independence (weak tail dependence), and to use this
approach to 
study the asymptotic behavior of tail
probabilities like
\begin{align}
\Pr[\min(X_1,\dots,X_n) \geq t ]\quad \text{as $t\to \infty$}\label{funcmin}
\end{align}
or 
\begin{align}
\Pr[X_1+\dots +X_n \leq  t ]\quad \text{as $t\to 0$ with
  $X_1\geq 0,\dots,X_n \geq 0$.}\label{funcsum}
\end{align}
when the components of the vector $(X_1,\dots,X_n)$ are asymptotically
independent. 

In multivariate extreme value theory, asymptotic
independence refers to the situation when
the probability that any two components of a random vector are
simultaneously large, on a suitable scale, is negligible compared to the probability that
any one component is large. This ensures that suitably renormalized
componentwise maxima of a sequence of independent copies of the vector become asymptotically independent. The components of an asymptotically
independent random vector may be dependent in the usual sense. For instance, the multivariate Gaussian distribution is known to
be asymptotically independent for any nondegenerate covariance
matrix \citep{sibuya1959bivariate}. 

Asymptotic independence is a natural property 
which is often observed in the data coming from many different application
domains
\citep{de1998sea,maulik2002asymptotic,draisma2004bivariate,ledford1997modelling,ledford1996statistics,heffernan2005hidden},
and is an inherent feature of many widely used models, for example, in
finance. In addition to
the already mentioned multivariate Gaussian, one can quote, e.g., the
multivariate generalized hyperbolic \citep{schlueter2012weak}, and more
generally all Gaussian mixture
models with exponentially decaying mixing variable (see Section
\ref{taildep.sec} below).  It is therefore important to understand the
implications of asymptotic independence assumption on the tail behavior of
random vectors. 

In the literature, asymptotic independence is often introduced using
the property of hidden regular variation
\citep{resnick2002hidden,maulik2004characterizations,resnick2007heavy},
which is a refinement of the coefficient of tail dependence introduced
in \cite{ledford1996statistics,ledford1997modelling}.  Recall that a random vector
$\bs X $ with values in $[0,\infty)^n$ is said to be multivariate
regularly varying if there exists a function $b(t)\to \infty$ as $t\to
\infty$ and a
non-negative Radon measure $\nu\neq 0$ such that 
\begin{align}
t \Pr\Big[\frac{\bs X}{b(t)}\in \cdot\Big] \to \nu\label{mrv}
\end{align}
as $t\to \infty$ in the sense of vague convergence of measures on the
cone $E = [0,\infty]^n \setminus \{0\}$. Now, $\bs X$ is said to possess the
property of hidden regular variation if in addition to \eqref{mrv},
there exists a non-decreasing function $b^*(t)\to \infty$ such that
$b(t)/b^*(t) \to \infty$ as $t\to \infty$ and a Radon measure
$\nu^*\neq 0$ such that 
\begin{align}
t \Pr\Big[\frac{\bs X}{b^*(t)}\in \cdot\Big] \to \nu^*\label{hrv}
\end{align}
on the cone $E^*:=E\setminus \cup_{i=1}^n L_i$, where $L_i$ is the
i-th coordinate axis. In other words, under hidden regular variation, the
probability $\Pr[X_i \geq x_i t, X_j \geq x_j t]$ for any $i\neq
j$ decays regularly, but at a faster rate than the tail probabilities
of individual components. 

The theory of hidden regular variation along with its more recent
extensions \citep{das2013living} allows to quantify the asymptotic
behavior of the tail
probabilities like \eqref{funcmin} or \eqref{funcsum} but it is well
suited for distributions with fat power-law tails and tail equivalent
margins and does not readily apply to, say, models with exponential
tail decay which are widely used in finance\footnote{In some cases,
  one can solve the problem of non-equivalent margins using the rank
  transform, which preserves the hidden regular variation
  \citep{heffernan2005hidden}, however, it does not preserve the
  structure of some tail functionals which one would like to compute
  in risk management applications.}. An alternative (but related)
approach to studying dependence within asymptotic independence,
consists in observing that for, say, identically
distributed positive random variables, asymptotic independence in the
lower tail implies that 
$$
\lim_{x\downarrow 0} \frac{\Pr[X_1\leq x,\dots,X_n \leq
  x]}{\Pr[X_1\leq x]} = 0,
$$
and information about ``residual'' dependence may be
extracted from the multivariate distribution function by studying a
related limit on the logarithmic scale. This leads to the so called
coefficient of weak tail dependence \citep{coles1999dependence,schlueter2012weak,hashorva2010residual,heffernan2000directory} usually defined (for the case of the lower index of a two-dimensional copula $C$) as
\begin{align}
\lim_{u\to 0} \frac{2\ln u}{\ln C(u,u)}-1. \label{tdcoef}
\end{align}
This coefficient is defined at the level of copulas, and hence does not require
strong assumptions on the margins.
It exists under much less stringent conditions than those required for
hidden regular variation (since here one is only interested in
log-scale asymptotics), but the knowledge of the weak tail dependence
coefficient does not provide information on the functionals like
\eqref{funcmin} or \eqref{funcsum}, which also depend on the margins. 

In this paper, we develop a functional extension of the weak tail dependence
coefficient \eqref{tdcoef}, which we term \emph{weak tail dependence
  function}. For the lower tail of the copula, the weak tail dependence function is
defined by 
\begin{align}
\chi(\lambda_1,\dots,\lambda_n) = \lim_{u\to
  0}\frac{\min_i \ln u^{\lambda_i}}{\ln C(u^{\lambda_1},\dots,u^{\lambda_n})},\quad
\lambda_1,\dots,\lambda_n \geq 0.  \label{chi.intro}
\end{align}
The extension of the weak tail dependence coefficient by the weak tail
dependence function is somewhat similar in spirit to the extension of
the strong tail dependence coefficient by the tail dependence function
\citep{joe10,kluppelberg2008semi,kluppelberg2007estimating}. 

We compute the weak tail
dependence function for commonly used families of copulas, notably for
the Gaussian copula and for all Gaussian mixture models with
exponentially decaying mixing variable, and show that
log-scale asymptotics of functionals like \eqref{funcmin} or
\eqref{funcsum} are expressed in terms of the weak tail dependence
function under weak assumptions on the margins. 
In particular (see Theorem \ref{main.thm}), if $X_1,\dots,X_n$ are random variables with
values in $(0,\infty)$, with 
survival functions $\bar F_1,\dots,\bar F_n$ and survival copula
$\overline C$, such that for $k=1,\dots,n$, 
$$
\ln \bar F_k(x) \sim \lambda_k \ln \bar F_0(x)
$$
as $x\to + \infty$ for some constants $\lambda_k$ and some
function $\bar F_0$, then 
$$
\ln \Pr[\min(X_1,\dots,X_n) \geq t ] \sim \frac{\min_i \ln \bar
  F_i(t)}{\bar \chi(\lambda_1,\dots,\lambda_n)},
$$
as $t\to +\infty$, where $\bar \chi$ is computed from the copula $\overline C$ using the
formula \eqref{chi.intro}. On the other hand (see Corollary
\ref{main2.thm}), 
 if $X_1,\dots,X_n$ are random variables with
values in $(0,\infty)$, with 
distribution functions $F_1,\dots,F_n$ and copula
$C$, such that for $k=1,\dots,n$, 
$\ln F_k(x)$ is slowly varying as $x\downarrow 0$ and 
$$
\ln F_k(x) \sim \lambda_k \ln  F_0(x)
$$
for some constants $\lambda_k$ and some
function $F_0$, then 
$$
\ln \Pr[X_1 + \dots + X_n \leq t ] \sim \frac{\min_i \ln 
  F_i(t)}{\chi(\lambda_1,\dots,\lambda_n)},
$$
as $t\downarrow0$. 

The assumption that the distribution functions are slowly varying includes all distributions with
  regularly varying left tail as well as parametric families such as
  log-normal, gamma, Weibull and many distributions from the financial
  mathematics literature. The assumption of asymptotic equivalence on
  the log scale ensures that the laws of components have similar
asymptotic behavior, but nevertheless is not very restrictive: for
example, different components can follow log-normal distributions with different
parameters, or have regularly varying tails with different indices.  Our method thus
provides less information than hidden regular variation (which allows
to compute the sharp asymptotics) but on the other hand is applicable in a much wider
context.

The rest of the paper is structured as follows. Section \ref{def.sec}
presents the definition and the basic properties of the weak
tail dependence function. The explicit form of the weak tail
dependence function for common copula families is given in Section
\ref{taildep.sec}. Finally, Section \ref{tailsum.sec} presents the
link between the weak tail dependence function and the asymptotics of tail probabilities like
\eqref{funcmin} and \eqref{funcsum} and illustrates the theory with an
example coming from financial mathematics.

\paragraph{Remarks on notation}
Throughout this paper, we write $f\sim g$ as $x$ tends to $a$ whenever
$$\lim_{x\to a}\frac{f(x)}{g(x)} = 1$$ and $f \lesssim g$ whenever $$\limsup_{x\to a}\frac{f(x)}{g(x)} \leq 1.$$  We recall that a function $f$ is
called slowly varying as $x$ tends to $0$ whenever $$\lim_{x\to
  0}\frac{f(\alpha x)}{f(x)} = 1$$ for all $\alpha >0$. 
Finally, we
define
$$
\Delta_n:= \{\bs w\in \mathbb R^n: w_i\geq 0, i=1,\dots,n, \sum_{i=1}^n
w_i = 1\}. 
$$

We also recall that the copula of a random vector
$(Y_1,\dots,Y_n)$ is a function $C:[0,1]^n: [0,1]$, satisfying the
assumptions
\begin{itemize}
\item $dC$ is a positive measure in the sense of Lebesgue-Stieltjes integration,
\item $C(u_1,\dots,u_n) = 0$ whenever $u_k=0$ for at least one $k$, 
\item $C(u_1,\dots,u_n) = u_k$ whenever $u_i = 1$ for all $i\neq k$,
\end{itemize}
and such that 
$$
\Pr[Y_1\leq y_1,\dots,Y_n \leq y_n] = C(\Pr[Y_1\leq
y_1],\dots,\Pr[Y_n\leq
y_n] ),\quad (y_1,\dots,y_n)\in \mathbb R^n.
$$
A copula exists by Sklar's theorem and is uniquely defined whenever
the marginal distributions of $Y_1,\dots,Y_n$ are continuous. We refer to
\cite{nelsen} for details on copulas. 

\section{Weak tail dependence function}
\label{def.sec}

\begin{definition}
The \emph{weak lower tail
  dependence function} $\chi(\lambda_1,\dots,\lambda_n)$ of a copula $C$
is defined by
\begin{align}
\chi(\lambda_1,\dots,\lambda_n) = \lim_{u\to
  0}\frac{\min_i\ln u^{\lambda_i}}{\ln C(u^{\lambda_1},\dots,u^{\lambda_n})},\label{defchi}
\end{align}
whenever the limit exists for all $\lambda_1,\dots,\lambda_n \geq 0$ such
that $\lambda_k>0$ for at least one $k$, with the standard convention that
$\ln 0 = \infty$ and ${1}/{\infty} = 0$. 
\end{definition}

\begin{remark}
In this paper, we focus on the lower tail of the copula, hence the term weak lower
tail dependence function. Weak tail
dependence functions for other tails of the copula may be defined in a
similar manner. For example, weak upper tail dependence function of a
copula $C$ is defined by
$$
\bar\chi(\lambda_1,\dots,\lambda_n) = \lim_{u\to
  0}\frac{\min_i\ln u^{\lambda_i}}{\ln \overline{C}(u^{\lambda_1},\dots,u^{\lambda_n})},
$$
where $\overline{C}$ is the survival copula which corresponds to
$C$. Properties of the weak upper tail dependence function and the form of this function can be easily deduced
from the properties and the form of the weak lower tail dependence
function given in this paper. 
\end{remark}

\begin{remark}
Assume that $C(u,\dots,u)> 0$ for $u>0$ and let $\widehat C(z_1,\dots,z_n) = C(e^{-z_1}, \dots, e^{-z_n})$ for
$z_1,\dots,z_n\geq 0$. Then \eqref{defchi} is equivalent to 
$$
\ln \widehat C (\lambda_1 z, \dots, \lambda_n z) \sim - 
\frac{\max_i \lambda_i}{\chi(\lambda_1,\dots,\lambda_n)} z \quad
\text{as $z\to +\infty$.}
$$
Therefore, existence of a weak tail dependence function is a kind of
multivariate regular variation property with index $1$ of the logarithm of the copula at the logarithmic
scale. Due to the presence of this log scale transformation, the
existence of the weak tail dependence function is not directly related
to the regular variation properties of the distribution at the
original scale, in particular, it is not implied by the hidden regular
variation property. 
\end{remark}
\paragraph{Properties of the weak lower tail dependence function}
The weak lower tail
  dependence function $\chi(\lambda_1,\dots,\lambda_n)$ of a copula
is order $0$ homogeneous: for all $r>0$,
  $$\chi(r\lambda_1,\dots,r\lambda_n) =  \chi(\lambda_1,\dots,\lambda_n).$$ 
It is increasing with respect to the concordance order of copulas and
admits the following bounds (the upper bound is due to the
  Frechet-Hoeffding upper bound on the copula):
$$
0\leq \chi(\lambda_1,\dots,\lambda_n) \leq 1.
$$
For the independence copula $C_{\perp}(u_1,\dots,u_n) = u_1\dots
u_n$, we get 
$$\chi(\lambda_1,\dots,\lambda_n) = \frac{\max_i \lambda_i}{\sum_i \lambda_i}.$$ 
The upper bound is attained for the complete dependence copula 
$C_{\parallel}(u_1,\dots,u_n) = \min(u_1,\dots,u_n)$. More
importantly, as shown by the following proposition, for any copula
with nonzero strong tail dependence coefficient in the lower tail, the weak lower tail dependence function
equals its upper bound. This measure of tail dependence is thus
relevant for distributions whose components are asymptotically
independent. Before stating the result, we recall the following
definition. 
\begin{definition}
The \emph{strong tail dependence coefficient} (for the lower tail) of a
copula $C$ is defined by
$$
\lambda_L = \lim_{u\downarrow 0} \frac{C(u,\dots,u)}{u},
$$ 
whenever the limit exists. When $\lambda_L>0$, the copula is said to have
the property of asymptotic dependence in the lower tail.
\end{definition}
\begin{proposition}\label{tailcop.prop}
Assume that a copula function $C$ has strong tail dependence coefficient
$\lambda_L>0$. Then, the weak lower tail dependence function of $C$
is equal to the upper bound:
$$
\chi(\lambda_1,\dots,\lambda_n) = 1,\quad \forall
\lambda_1,\dots,\lambda_n\geq 0. 
$$
\end{proposition}
\begin{proof}
From the definition of $\lambda_L$, for any $\varepsilon>0$ and $u$
sufficiently small, 
$$
C(u,\dots,u) \geq (\lambda_L - \varepsilon) u.
$$ 
Using the fact that the copula is increasing in each argument, we
have, for $u$ sufficiently small,
$$
\frac{\ln C(u^{\lambda_1},\dots,u^{\lambda_n})}{\ln u} \leq
\frac{\ln (\lambda_L -
    \varepsilon)+ \max(\lambda_1,\dots,\lambda_n)\ln u }{\ln u},
$$
which shows that 
$$
\limsup_{u\downarrow 0} \frac{\ln
  C(u^{\lambda_1},\dots,u^{\lambda_n})}{\ln u}  = \max(\lambda_1,\dots,\lambda_n).
$$
Combining this with the Frechet-Hoeffding upper bound on the copula,
the proof is complete. 
\end{proof}

Strong tail dependence coefficients for different
copula families are listed, for instance, in \cite{nelsen,heffernan2000directory}. In
particular, it is known that the Gaussian copula has the property of
asymptotic independence \citep{sibuya1959bivariate}. 
By contrast, all copulas of
elliptical distributions with regularly varying tails, including, in
particular, the $t$-copula, are known to have the property of
asymptotic dependence \citep{hult2002multivariate}, and therefore, for
these copulas the weak tail dependence
function equals $1$.

\section{Weak lower tail dependence function for common copula
  families}
\label{taildep.sec}

\paragraph{Gaussian copula}The
Gaussian copula with correlation matrix $R$ is the unique copula of any
Gaussian vector with correlation matrix $R$ and nonconstant components (it does not depend on the
mean vector and on the variances of the components). The following
proposition characterizes the weak lower tail dependence function of
the Gaussian copula. 
\begin{proposition}\label{gausstail.prop}
Let $C$ be an $n$-dimensional Gaussian copula with correlation matrix $R$ with $\det R
\neq 0$. Then,
$$
\chi(\lambda_1,\dots,\lambda_n) = \max_i \lambda_i \min_{\bs w\in \Delta_n}
\bs w^\top \Sigma \bs w,\quad \text{for all $\lambda_1,\dots,\lambda_n>0$,}
$$
where the matrix $\Sigma$ has coefficients $\Sigma_{ij} =
\frac{R_{ij}}{\sqrt{\lambda_i\lambda_j}}$, $1\leq i,j\leq n$. 
\end{proposition}
\begin{proof}
Let $\bs X = (X_1,\dots,X_n)$ be a centered Gaussian vector with covariance
matrix $\Sigma$ defined above. The proof is based on the following
lemma. 
\begin{lemma}
Let $\mathbf{X}$ be an $n$-dimensional centered Gaussian vector with covariance matrix
$\Sigma$ assumed to be nondegenerate. 
Then there exist positive constants $c$ and $C$ and an integer $\bar
n$ with $1\leq \bar n\leq n$ such that, for all $z$
sufficiently large,
\begin{align}
\frac{c}{|z|^{\bar n}} e^{-\frac{z^2}{2 \min_{\bs w\in \Delta_n} \bs w^\top \Sigma \bs w}}\leq \Pr[X_1 \geq z,\dots,X_n \geq
z] \leq\frac{C}{|z|^{\bar n}} e^{-\frac{z^2}{2 \min_{\bs w\in \Delta_n} \bs w^\top \Sigma \bs w}}.\label{gaussbound}
\end{align}
\end{lemma}
\begin{proof}
The proof is based on the estimates of multivariate Gaussian tails given in
\cite{hashorva2003multivariate}. Taking $\mathbf{t} = z\bs 1$ and
using the notation introduced in Proposition 2.1 of this reference, we
see that (i) $|I_t|$ does not depend on $z$ and we set $\bar n =
|I_t|$; (ii) for every $i\in I_t$, $h_i = c_i z$ for some constant
$c_i>0$; (iii) the function $R(t)$ defined in \cite[equation
(1.2)]{hashorva2003multivariate} is equivalent to ${1}/{t}$ as
$t\to +\infty$; and finally (iv) the constant $\alpha_t$ satisfies
$$
\alpha_t = \min_{\mathbf x \geq \mathbf t}\bs x^\top \Sigma^{-1} \bs x =
z^2 \min_{\mathbf x \geq \mathbf 1}\bs x^\top \Sigma^{-1} \bs x.
$$
Using the method of Lagrange multipliers, we further get
\begin{align*}
\min_{\mathbf x \geq \mathbf 1}\bs x^\top \Sigma^{-1} \bs x &= \max_{
  {\bs \lambda\geq 0} }\min_{\mathbf x\in \mathbb R^n} \bs x^\top \Sigma^{-1}
\bs x - \bs \lambda^\top (\bs x-\bs 1) = \max_{\bs \lambda\geq 0}
-\frac{1}{4}\bs \lambda^\top \Sigma \bs \lambda + \bs \lambda^\top \bs 1\\ &=
\max_{\rho \geq 0, \bs w \in \Delta_n}  -\frac{\rho^2}{4}\bs w^\top
\Sigma \bs w + \rho = \max_{\bs w\in \Delta_n}\frac{1}{\bs w^\perp
  \Sigma \bs w};
\end{align*}
so that 
$$
\alpha_t = \frac{z^2}{\min_{\bs w\in \Delta_n}\bs w^\perp
  \Sigma \bs w}.
$$
Both upper and lower bounds in \eqref{gaussbound} then follow from formula (3.8)
in \cite{hashorva2003multivariate}. 
\end{proof}
From the above lemma, using the symmetry of centered Gaussian vectors,
we deduce that 
$$
\ln \Pr[X_1 \leq z,\dots,X_n \leq
z]\sim -\frac{z^2}{2 \inf_{\bs w\in \Delta_n} \bs w^\top \Sigma \bs w}
$$
as $z$ tends to $-\infty$. 
Applying this to a
single Gaussian variable yields $$\ln\Pr[X_i\leq z] \sim -\frac{z^2
  \lambda_i}{2},\quad z\to \infty.$$
Now combine these estimates to get, for
$\varepsilon$ and $z$ small enough,
\begin{align*}
  -\frac{z^2(1+\varepsilon)}{2 \inf_{\bs w\in \Delta_n} \bs w^\top \Sigma\bs  w}&\leq \ln \Pr[X_1 \leq z,\dots,X_n \leq
z] = \ln C(\Pr[X_1\leq z],\dots,\Pr[X_n\leq z])\\
&\leq \ln C(e^{-\frac{z^2 \lambda_1 (1-\varepsilon)}{2}},\dots,e^{-\frac{z^2 \lambda_n (1-\varepsilon)}{2}}).
\end{align*}
Letting $u =
e^{-{z^2(1-\varepsilon)}/{2}}$, this leads to
$$
\frac{1+\varepsilon}{(1-\varepsilon)\inf_{\bs w\in \Delta_n}\bs  w^\top \Sigma
    \bs w}\ln u\leq \ln C(u^{\lambda_1},\dots,u^{\lambda_n}).
$$
Dividing by $\min_i \ln u^\lambda$, and using the fact that
$\varepsilon$ is arbitrary, we finally get
$$
{\max_i \lambda_i\inf_{\bs w\in \Delta_n} \bs w^\top \Sigma
    \bs w}\leq \limsup_{u\to 0} \frac {\min_i \ln u^\lambda}{\ln C(u^{\lambda_1},\dots,u^{\lambda_n})}.
$$
The upper bound may be obtained in a similar fashion. 
\end{proof}

\paragraph{Gaussian mixtures with exponentially decaying mixing
  variable} Our next result describes the marginal tail behavior and
the weak lower tail dependence function of Gaussian mean-variance mixtures.
\begin{proposition}\label{mixture}
Let $\bs Y = (Y_1,\dots,Y_n)^\top$ be a centered nondegenerate Gaussian vector with correlation matrix $R$, and let $\bs\mu\in \mathbb R^n$, $\sigma_i =
\sqrt{\mathrm{Var}\, Y_i}$ for $i=1,\dots,n$ and $\tilde \mu_i =
{\mu_i}/{\sigma_i}$ for $i=1,\dots,n$.
Assume that $Z$ is a positive random variable with density
  $\rho(s)$ satisfying
$$
\rho(s) = e^{-\theta s + o(s)},\quad s\to \infty
$$
with $\theta>0$. 
Let $\bs X$ be defined by $\bs X = \sqrt{Z} \bs Y + Z\bs\mu$.  Then
\begin{itemize}
\item For $i=1,\dots,n$, 
$$
\ln \Pr[X_i \leq x]
\sim\frac{2\theta}{\sqrt{2\theta\sigma_i^2 + \mu_i^2} - \mu_i}
x,\quad x\to -\infty. 
$$

\item The copula of $\bs X$ has weak lower tail dependence function
$$
\chi(\lambda_1,\dots,\lambda_n) = \max_i \lambda_i \min_{\bs v}
\Big\{\sqrt{2\theta \bs v^\top R \bs v + (\tilde{\bs\mu}^\top \bs v)^2} -
  \tilde{\bs\mu}^\top \bs v\Big\},
$$
where the minimum is taken over the set
$$
\big\{\bs v\in \mathbb R^n, \ v_i\geq 0,i=1,\dots,n,\ \sum_{i=1}^n v_i
\lambda_i (\sqrt{2\theta  + \tilde\mu_i^2} - \tilde\mu_i)\leq 1\big\}.
$$
\end{itemize} 
\end{proposition}
Remark that in the general case, the weak lower tail dependence
function of a Gaussian mixture may depend on the correlation matrix
$R$, the normalized mean vector $\tilde{\bs \mu}$ and the decay rate $\theta$, since all these parameters affect the
dependence structure of the random vector. However, in the symmetric
case $(\bs\mu=0)$, it is easy to see that the weak lower tail dependence
function depends only on the correlation matrix.
\begin{corollary}\label{mixturecor}
Let $\bs X = \sqrt{Z} \bs Y$ where $\bs Y$ is centered Gaussian vector with correlation
matrix $R$, assumed to be nondegenerate, and $Z$ satisfies the assumption of Proposition \ref{mixture}. Then,
$$
\chi(\lambda_1,\dots,\lambda_n) = \max_i \lambda_i {\min\limits_{\bs w\in \Delta_n}
  \sqrt{ \bs w^\top \Sigma \bs w } },
$$
where the matrix $\Sigma$ has coefficients $$\Sigma_{ij} =
\frac{R_{ij}}{\lambda_i\lambda_j}.$$ 
\end{corollary}

\begin{remark}
Proposition \ref{mixture} and Corollary
\ref{mixturecor} 
improve our understanding of the tail dependence of Gaussian mixture
models with exponential decay of the mixing variable. For
example, taking $\bs \mu=0$, we have 
$$
\chi(1,\dots,1) = \min_{\bs w\in \Delta_n} \sqrt{\bs w^\top R \bs w}< 1
$$
whenever the correlation matrix $R$ is nongenenerate. Therefore, by
Proposition \ref{tailcop.prop} we conclude that \emph{Gaussian variance
mixture models with exponentially decaying mixing variable have no
strong tail dependence.} In particular, for $n=2$, 
$$
R = \left(\begin{aligned}&1 && \rho \\ &\rho &&
    1\end{aligned}\right)\quad \text{and}\quad 
\chi(1,1) = \sqrt{\frac{1+\rho}{2}},
$$
and we recover and extend the main result of \cite{schlueter2012weak},
where this value has been computed for the generalized hyperbolic
distribution. More precisely, in this reference, the weak tail dependence
  coefficient is defined (for the left tail) as $$\lim_{u \to 0}
  \frac{2\ln(u)}{\ln C(u,u)}-1,$$ which corresponds to $2\chi(1,1)-1$
  in our notation, and is found to be equal to
  $$2\sqrt{\frac{1+\rho}{2}} - 1.$$ 

\end{remark}
The proof of Proposition \ref{mixture} is based on the following
estimates which can be found in \cite{gulisashvili.tankov.14a}.
\begin{lemma}\label{gt}
Let $\bs Y$ be a centered Gaussian vector with a nondegenerate covariance
matrix $\cm$, and let $\bs \mu \in \mathbb R^n$. Suppose that $Z$ is a random variable with
values in $(0,\infty)$ admitting a density $\rho$.
\begin{itemize}
\item Assume that $\rho(s)\leq
c_1 e^{-\theta s}$ for $s\geq 1$, where $\theta
>0$ and $c_1>0$ are constants. Then, there
exists $C_1> 0$ such that for $k$ sufficiently large,
$$
\Pr\Big[\sum_{i=1}^n e^{Y_i \sqrt{Z} + \mu_i Z } \leq e^{-k}\Big]
\leq C_1 e^{-c_\theta^*k },
$$
where 
\begin{equation}
c_\theta^* = \min_{t\geq 0} \max_{\bs w\in \Delta_n} \Big\{\theta t +
  \frac{(1+t\bs \mu^\top \bs w)^2}{2\bs w^\top \cm \bs w t}\Big\} = \max_{\bs w\in
\Delta_n} \frac{2\theta}{\sqrt{2\theta \bs w^\top \cm \bs w + (\bs \mu^\top \bs w)^2} - \bs \mu^\top
\bs w}.
\label{E:prev}
\end{equation}
\item Assume that 
$\rho(s)\geq
c_2 e^{-\theta s}$ for $s\geq 1$, where $\theta
>0$ and  $c_2>0$ are constants. Then, there
exists $C_2>0$ such that for $k$ sufficiently large,
$$
\Pr\Big[\sum_{i=1}^n e^{Y_i \sqrt{Z} + \mu_i Z } \leq e^{-k}\Big] \geq
C_2 k^{-n}e^{-c_\theta^* k},
$$
\end{itemize}
\end{lemma}

\begin{proof}[Proof of Proposition \ref{mixture}]
Under the assumptions of Proposition \ref{mixture}, for every
$\varepsilon>0$, one can find constants $c_1>0$ and $c_2>0$ such that 
$$
c_1 e^{-(\theta+\varepsilon) s}\leq \rho(s) \leq c_2
e^{-(\theta-\varepsilon) s}, \quad s\geq 1.
$$
Using the bounds of Lemma \ref{gt} and taking the logarithm yields,
for $x$ small enough,
$$
\ln C_2 - n\ln \ln \frac{1}{x} +c^*_{\theta+\varepsilon} \ln x\leq \ln \Pr\Big[\sum_{i=1}^n e^{X_i } \leq x\Big] \leq
\ln C_1 +c^*_{\theta-\varepsilon} \ln x.
$$
Divide by $\ln x$ and pass to the limit $x\to 0$ to get
\begin{align*}
&c^*_{\theta+\varepsilon} \geq \lim\sup_{x\to 0} \frac{\ln \Pr\big[\sum_{i=1}^n e^{X_i } \leq x\big]}{\ln x} \\ &\lim\inf_{x\to 0} \frac{\ln \Pr\big[\sum_{i=1}^n e^{X_i } \leq x\big]}{\ln x} \geq
c^*_{\theta-\varepsilon}.
\end{align*}
Since $c^*_\theta$ is obviously continuous in $\theta$ and
$\varepsilon$ is arbitrary, we conclude
that 
$$
\lim_{x\to 0} \frac{\ln \Pr\big[\sum_{i=1}^n e^{X_i } \leq x\big]}{\ln x} = c^*_\theta.
$$
Applying this result to a single component $X_i$, we get
$$
\lim_{x\to 0} \frac{\ln \Pr[e^{X_i } \leq x]}{\ln x} = \frac{2\theta}{\sqrt{2\theta\sigma_i^2 + \mu_i^2} - \mu_i}.
$$
Therefore, $\ln \Pr[e^{X_i } \leq x]$ is slowly varying as $x$
tends to $0$, and by
Theorem \ref{main.thm}, 
$$
\chi(\lambda_1,\dots,\lambda_n) = \frac{\max_i
  \lambda_i}{c^*_\theta}\quad \text{for}\quad \lambda_i = \frac{2\theta}{\sqrt{2\theta\sigma_i^2 + \mu_i^2} - \mu_i}.
$$
However, since $\chi$ depends only on the copula, it is invariant with
respect to the transformation $\mu_i \mapsto \alpha_i\mu_i$ and
$\sigma_i\mapsto \alpha_i\sigma_i$ for $i=1,\dots,n$ for any vector
$\bs \alpha\in \mathbb R^n$ with positive components. Hence, for arbitrary
$\lambda_i>0$, one can always find $\alpha_i>0$ such that 
$$
\lambda_i = \frac{2\theta}{\sqrt{2\theta (\alpha_i\sigma_i)^2 +
    (\alpha_i\mu_i)^2} - \alpha_i\mu_i}.
$$
To complete the proof, substitute this into the expression for
$c^*_\theta$ and make the change of variable $$v_i = \frac{w_i
  \alpha_i \sigma_i}{2\theta}$$ in the optimization problem. 
\end{proof}

\paragraph{Archimedean copulas} Recall that given a function $\phi: [0,1]\to [0,\infty]$
which is continuous, strictly decreasing and such that its inverse
$\phi^{-1}$ is completely monotonic, the Archimedean copula with generator $\phi$ is
defined by
$$
C(u_1,\dots,u_n) = \phi^{-1}\{\phi(u_1)+\dots+\phi(u_n)\}. 
$$
The following simple result gives the weak lower tail dependence
function for an Archimedean copula. The case when $\ln \phi^{-1}$ is
regularly varying includes for example the Gumbel copula with
$\phi^{-1}(t) = \exp(-t^{1/\theta})$ and several other families. 
\begin{proposition}
Let $C$ be an Archimedean copula with generator function $\phi$. 
\begin{itemize} 
\item[(i).] If $\ln \phi^{-1}$ is regularly varying at $+\infty$ with index
$\alpha>0$, then,
$$
\chi(\lambda_1,\dots,\lambda_n) = \frac{\max(\lambda_1,\dots,\lambda_n)}{\big(\lambda_1^{1/\alpha}+\dots + \lambda_n^{1/\alpha}\big)^\alpha}
$$
\item[(ii).] If $\ln \phi^{-1}$ is slowly varying at $+\infty$, then
$$
\chi(\lambda_1,\dots,\lambda_n) = 1
$$
\end{itemize}
\end{proposition}
\begin{remark}
The condition that $\ln \phi^{-1}$ be regularly varying at $0$ is
sufficient for $C$ to be in the max-domain of attraction of the Gumbel
copula \citep{genest1989characterization}. However, for the
existence of the weak lower tail dependence function we require that
$\ln \phi^{-1}$ be regularly varying at $+\infty$ which is a
different condition. 
\end{remark}
\begin{remark}
When $\ln \phi^{-1}$ is regularly varying but not slowly varying at $+\infty$, Proposition \ref{tailcop.prop}
 implies that the copula $C$ has no strong dependence in the left
 tail, meaning that the strong tail dependence coefficient $\lambda_L$
 equals zero. When $\ln \phi^{-1}$ is slowly varying, the situation
 is less clear. For an Archimedean copula, the strong tail dependence
 coefficient is given by
$$
\lambda_L = \lim_{u\downarrow 0} \frac{C(u,\dots,u)}{u} =
\lim_{u\downarrow 0} \frac{\phi^{-1}\{\phi(u_1)+\dots+\phi(u_n)\}}{u} =
\lim_{t\to \infty} \frac{\phi^{-1}(nt)}{\phi^{-1}(t)}. 
$$
Therefore, when $\phi^{-1}$ is slowly or regularly varying at
$+\infty$, $\lambda_L$ exists and is strictly positive, and so $\chi$
attains its upper bound
$\chi(\lambda_1,\dots,\lambda_n)=1$ for all
$\lambda_1,\dots,\lambda_n\geq 0$. However, there exist situations
when $\lambda_L=0$ yet  $\chi(\lambda_1,\dots,\lambda_n)=1$. Indeed,
the function $$\phi^{-1}(u) = e^{-\left\{\ln(1+u) +
    \frac{1}{2}\right\}^2+\frac{1}{4}}$$ is a valid inverse generator
function of an Archimedean copula in dimension $2$ and is rapidly varying at $+\infty$
(which means that $\lambda_L=0$) but $\ln \phi^{-1}$ is slowly
varying. 
\end{remark}
\begin{proof}
Assume first that $\ln \phi^{-1}$ is regularly varying with index $\alpha>0$. By definition of $\chi$,
\begin{align*}
\chi(\lambda_1,\dots,\lambda_n) &= \lim_{u\to 0}
\frac{\max(\lambda_1,\dots,\lambda_n)\ln u}{\ln
  \phi^{-1}\{\phi(u^{\lambda_1})+\dots+\phi(u^{\lambda_n})\}}\\
& = \lim_{u\to 0}
\frac{\max(\lambda_1,\dots,\lambda_n)\ln \phi^{-1}(\phi(u))}{\ln
  \phi^{-1}\{\phi(e^{\lambda_1\ln u})+\dots+\phi(e^{\lambda_n\ln
    u})\}}
\end{align*}
By the inversion theorem for regularly varying functions
\citep{bingham1989regular}, the function $u\mapsto\phi(e^u)$ is
regularly varying at $-\infty$ with index
${1}/{\alpha}$. Therefore, for any $\varepsilon>0$ and $u$
sufficiently small,
\begin{align*}
(1-\varepsilon) (\lambda_1^{1/\alpha}+\dots +
\lambda_n^{1/\alpha})\phi(u)&\leq \phi(e^{\lambda_1\ln
  u})+\dots+\phi(e^{\lambda_n\ln u}) \\ &\leq (1+\varepsilon) (\lambda_1^{1/\alpha}+\dots + \lambda_n^{1/\alpha})\phi(u),
\end{align*}
and we conclude using the regular variation of $\ln \phi^{-1}$ and
the fact that $\varepsilon$ is arbitrary. The proof for the case when
$\ln \phi^{-1}$ is slowly varying is similar.

\end{proof}

\paragraph{Extreme value copulas}
The weak lower tail dependence function can be alternatively
represented as follows.
\begin{align}
\chi(\lambda_1,\dots,\lambda_n) = -\frac{\max_i \lambda_i}{\ln
  \lim_{t\to \infty} C\{(e^{-\lambda_1})^{t},\dots,(e^{-\lambda_n})^{t}\}^{\frac{1}{t}}}.\label{ev.eq}
\end{align}
Let $C$ be an extreme value copula 
\cite[chapter 6]{de2007extreme}, that is, a copula
satisfying 
$$
C(u_1^{1/m},\dots,u_n^{1/m})^m = C(u_1,\dots,u_n),\quad m\in \mathbb N^*,\quad
(u_1,\dots,u_n)\in [0,1]^n,
$$
where $\mathbb N^*$ denotes the set of natural numbers excluding
zero. 
From \eqref{ev.eq} it follows that the weak lower tail dependence
function of $C$ is given simply by
$$
\chi(\lambda_1,\dots,\lambda_n) = -\frac{\max_i \lambda_i}{\ln
 C(e^{-\lambda_1},\dots,e^{-\lambda_n})}.
$$

\section{Tail asymptotics of weakly dependent random vectors}
\label{tailsum.sec}
In this section we show how the weak tail dependence function may be
used to characterize the log-scale tail behavior of certain functionals of
components of weakly dependent random vectors.

Our first example shows that under relatively weak assumptions on the
margins, the log-scale asymptotic behavior of the tails of the distribution function of a weakly dependent
random vector may be deduced from the weak tail dependence function.
\begin{theorem}\label{main.thm}${}$\begin{itemize}\item[(i)]
Let $X_1,\dots,X_n$ be random variables with values in $(a,b)$, where
$b \in \mathbb R \cup \{+\infty\}$ with marginal survival functions $\bar F_1,\dots,\bar F_n$ and
survival copula
$\overline C$ satisfying the following assumptions.
\begin{itemize}
\item For each $k=1,\dots,n$,  $$
\ln \bar F_k(x) \sim \lambda_k \ln \bar F_0(x)\quad \text{as $x \uparrow b$}
$$
for some constants $\lambda_k >0 $ and some function $\bar F_0$. 
\item The copula $\overline C$ admits a weak upper tail dependence function
  $\bar\chi$. 
\end{itemize}
Then,
$$
\lim_{x\uparrow b} \frac{\ln \Pr[\min(X_1,\dots,X_n) \geq x]}{\min_i\ln \Pr[X_i
  \geq x]} =
\frac{1}{\bar\chi(\lambda_1,\dots,\lambda_n)}. 
$$
\item[(ii)] Let $X_1,\dots,X_n$ be random variables with values in $(a,b)$, where
$a \in \mathbb R \cup \{-\infty\}$ with marginal distribution functions $F_1,\dots,F_n$ and
copula
$C$ satisfying the following assumptions.
\begin{itemize}
\item For each $k=1,\dots,n$,  $$
\ln F_k(x) \sim \lambda_k \ln \bar F_0(x)\quad \text{as $x \downarrow a$}
$$
for some constants $\lambda_k >0 $ and some function $\bar F_0$. 
\item The copula $C$ admits a weak lower tail dependence function
  $\chi$. 
\end{itemize}
Then,
$$
\lim_{x\downarrow a} \frac{\ln \Pr[\max(X_1,\dots,X_n) \leq x]}{\min_i\ln \Pr[X_i
  \leq x]} =
\frac{1}{\chi(\lambda_1,\dots,\lambda_n)}. 
$$
\end{itemize}
\end{theorem}

\begin{proof}We prove only the first part, the proof of the second
  part being very similar.
First, observe that 
\begin{align*}
\Pr[\min(X_1,\dots,X_n) \geq x] =    \Pr[X_1\geq x,\dots,X_n \geq x] = \overline C\{\bar
  F_1(x),\dots,\bar F_n(x)\}. 
\end{align*}
By assumption of the theorem, for any $\varepsilon>0$ and $x$ close
enough to $b$, 
$$
\bar F_0(x)^{\lambda_k (1+\varepsilon)}  \leq \bar F_k(x)
\leq \bar F_0(x)^{\lambda_k (1-\varepsilon)},\quad k=1,\dots,n.
$$
Therefore, 
\begin{align*}
\overline C\{\bar F_0(x)^{\lambda_1
  (1+\varepsilon)},\dots,\bar F_0(x)^{\lambda_n (1+\varepsilon)}\}&\leq
\Pr[\min(X_1,\dots,X_n) \geq x] \\ &\leq \overline C\{\bar F_0(x)^{\lambda_1
  (1-\varepsilon)},\dots,\bar F_0(x)^{\lambda_n (1-\varepsilon)}\}
\end{align*}
and by definition of the weak lower tail dependence function, for $x$
close enough to $b$ enough, we then
have
\begin{align*}
\bar F_0(x)^{\bar\chi^{-1}(\lambda_1,\dots,\lambda_n)(1+\varepsilon)^2
  \max_i \lambda_i}&\leq \Pr[\min(X_1,\dots,X_n) \geq x] \\ &\leq
  \bar F_0(x)^{\bar\chi^{-1}(\lambda_1,\dots,\lambda_n)(1-\varepsilon)^2 \max_i \lambda_i}.
\end{align*}
Taking the logarithms and using the fact that $\varepsilon$ is
arbitrary shows that
$$
\lim_{x\uparrow b}\frac{\ln \Pr[\min(X_1,\dots,X_n) \geq
  x]}{\max_i \lambda_i\ln \bar F_0(x)} = \bar\chi^{-1}(\lambda_1,\dots,\lambda_n)
$$
and therefore
$$
\lim_{x\uparrow b}\frac{\ln \Pr[\min(X_1,\dots,X_n) \geq x]}{\ln \min_i \Pr[X_i \geq x]} =
{\bar\chi^{-1}(\lambda_1,\dots,\lambda_n)}.
$$
\end{proof}

Under the assumption of slow variation on the log scale of the
marginal distribution functions, the same asymptotic behavior extends
to more complex functionals of the random vector. 
\begin{corollary}\label{main2.thm} ${}$
\begin{itemize}
\item Let $A \subset [0,\infty)^n$ be a measurable set such that there exist $0<k<K<\infty$ with
  $[K,\infty)^n \subset A \subset [k,\infty)^n$ and let $X_1,\dots,X_n$ be random variables with values in $(0,\infty)$
with marginal survival functions $\overline F_1,\dots,\overline F_n$
and survival copula
$\overline C$ satisfying the following assumptions.
\begin{itemize}
\item For each $k=1,\dots,n$,  $\ln \overline F_k$ is slowly varying at
$+\infty$ and satisfies 
$$
\ln \overline F_k(x) \sim \lambda_k \ln  \overline F_0(x)\quad
\text{as $x\uparrow +\infty$}
$$
for some constant $\lambda_k >0 $ and some function $\overline F_0$. 
\item The copula $\overline C$ admits a weak upper tail dependence function
  $\bar\chi$. 
\end{itemize}
Then,
$$
\lim_{x\uparrow +\infty} \frac{\ln \Pr[(X_1,\dots,X_n) \in x A]}{\min_i\ln \Pr[X_i
  \geq x]} =
\frac{1}{\chi(\lambda_1,\dots,\lambda_n)}. 
$$
\item Let $A \subset [0,\infty)^n$ be a
  bounded measurable set such that there exist $0<k<K<\infty$ with
  $[0,k]^n \subset A \subset [0,K]^n$ and let $X_1,\dots,X_n$ be random variables with values in $(0,\infty)$
with marginal distribution functions $F_1,\dots,F_n$ and copula
$C$ satisfying the following assumptions.
\begin{itemize}
\item For each $k=1,\dots,n$,  $\ln F_k$ is slowly varying at
zero and satisfies 
$$
\ln F_k(x) \sim \lambda_k \ln F_0(x)\quad \text{as $x\downarrow 0$}
$$
for some constant $\lambda_k >0 $ and some function $F_0$. 
\item The copula $C$ admits a weak lower tail dependence function
  $\chi$. 
\end{itemize}
Then,
$$
\lim_{x\downarrow 0} \frac{\ln \Pr[(X_1,\dots,X_n) \in x A]}{\min_i\ln \Pr[X_i
  \leq x]} =
\frac{1}{\chi(\lambda_1,\dots,\lambda_n)}. 
$$
\end{itemize}
\end{corollary}
\begin{remark}
Taking $A = \{\bs x\in \mathbb R^n: x_i\geq 0, i=1,\dots,n, \sum x_i \leq
1\}$ in the second part, one can, for instance, compute the asymptotics
of $\Pr[X_1+\dots+X_n\leq x]$ as $x\to 0$.
 
The assumption on the marginal distributions
covers, e.g., distributions which are regularly varying at zero as
well as those which are slowly varying at zero. It excludes
  distributions with very fast decay at zero, such as the normal inverse
  Gaussian. Note that when $F_0$ is regularly varying, one can relax
  the assumptions on $A$ and only assume that $\overline A \in
  (0,\infty)^n$ in the first part and that $A$ is bounded in the
  second part. 
\end{remark}
\begin{proof}
The proof of the two parts being very similar, we focus on the second
part of the corollary. By the assumptions on $A$, 
\begin{align*}
\Pr[\max(X_1,\dots,X_n)\leq xk ]&\leq 
 \Pr[(X_1,\dots,X_n) \in xA] \\ &\leq \Pr[\max(X_1,\dots,X_n)\leq xK ].
\end{align*}
On the other hand, since $\ln F_0(x)$ is slowly varying at zero,
$$
\ln \Pr[X_i \leq Kx] \sim \ln \Pr[X_i \leq kx] \sim \ln \Pr[X_i \leq x]
$$
as $x\to 0$. 

\end{proof}



\paragraph{Example}\label{app.sec}
In this example we show how the asymptotic results obtained in this
note may be used to analyze the tail behavior of a portfolio of
options in the multidimensional Black-Scholes model. It should be
emphasized that the multidimensional Black-Scholes model does not
provide an adequate description of market movements in times of market
stress \citep{mcneil2010quantitative}. Nevertheless, this model, and
more generally the multivariate Gaussian distribution is still widely
used by practitioners for day-to-day risk management and it is
therefore important to understand
the tail behavior of portfolios in this model. 



Fix a time horizon $T$ and let $(X_1,\dots,X_n)$ denote the vector of logarithmic returns of $n$
risky assets over this time horizon. The asset
prices at date $T$ are then given
by $S_i = e^{X_i}$ for $i=1,\dots,n$ where we have assumed without
loss of generality that the initial values of all assets are normalized
to $1$. We suppose that the $n$ risky assets follow the
multidimensional Black-Scholes model. This means that the distribution
of the vector $(X_1,\dots,X_n)$ is Gaussian, and we denote by $\cm
T$ its covariance matrix and by $\bs \mu T$ its mean vector. 

We are interested in the tail behavior of a long-only portfolio of
European call options written on $n$ risky assets. To simplify
the discussion we assume that the portfolio contains exactly one
option on each of the risky assets, but the setting can obviously be
extended to an arbitrary number of options. The log-strikes of the options
will be denoted by $(k_1,\dots,k_n)$ and the maturity dates by
$(T_1,\dots,T_n)$, where $T_i > T$ for $i=1,\dots,n$. Assuming that
the interest rate is zero, the price of $i$-th option at date $T$ is
given by the Black-Scholes formula \citep{blackscholes}:
$$
P_i = e^{X_i} \mathcal N(d_+)- e^{k_i} \mathcal N(d_-),\quad d_{\pm} = \frac{X_i -
  k_i}{\sigma_i \sqrt{T_i-T}}\pm \frac{\sigma_i \sqrt{T_i -
    T}}{2},\quad \sigma_i = \sqrt{\cm_{ii}},
$$
where $\mathcal N$ is the standard normal distribution function. 

For a real-world risk management application it would of course be too
naive to assume that the volatility $\sigma_i$, which is used to price
the option, is constant and equal to $\sqrt{\cm_{ii}}$. In practice one
needs either to assume a multivariate Gaussian distribution for both
stock returns and volatilities, or to introduce the so-called implied
volatility skew, that is, assume that $\sigma_i$ is a (typically
decreasing) deterministic function of $S_i$. This example should
therefore be seen as a toy example whose main purpose is to illustrate
and motivate the theory of the paper. The
development of a full-scale risk management application of this theory
is left to further research. 

The following proposition clarifies the asymptotic behavior of the probability
$\Pr[P_1+ \dots + P_n \leq z]$ as $z$ tends to $0$. It is
surprising that even though the tails of asset returns are very thin
(Gaussian) in the Black-Scholes model, the distribution of a portfolio
of options has power-law tails. This reflects the fact that options
are much more risky than stocks. 
\begin{proposition}\label{bsopt.prop}
As $z$ tends to $0$, 
$$
\ln \Pr[P_1+ \dots + P_n \leq z]\sim  \frac{\ln
  z}{\inf_{\bs w\in \Delta_n} \bs w^\top \Sigma \bs w},
$$
where $\Sigma$ is a $n\times n$ matrix with elements given by $$\Sigma_{ij} =
\frac{\cm_{ij} T}{\sigma_i \sigma_j \sqrt{(T_i-T)(T_j-T)}}.$$ 
\end{proposition}
\begin{proof} $P_1,\dots,P_n$ are obviously increasing and continuous
  functions of the Gaussian random variables
  $(X_1,\dots,X_n)$. Therefore, the copula of $(P_1,\dots,P_n)$ is the
  Gaussian copula with correlation matrix with elements
  $R_{ij} = \frac{\cm_{ij}}{\sigma_i\sigma_j}$. It remains to
  characterize the asymptotic behavior of the distribution functions
  of $P_1,\dots,P_n$.

Let $$\tilde X_i = \frac{X_i - \mu_i T}{\sigma_i\sqrt{T}}$$ for $i=1,\dots,n$ and define
\begin{align*}
f_i(x) &= e^{\mu_i T + x\sigma_i \sqrt{T}} \mathcal N\{d_+(x)\} - e^{k_i} \mathcal N\{d_-(x)\},\\d_{\pm}(x) &=
x \sqrt{\frac{T}{T_i-T}}- \frac{ \mu_i T +
  k_i}{\sigma_i \sqrt{T_i-T}}\pm \frac{\sigma_i \sqrt{T_i -
    T}}{2}.
\end{align*}
Then, $\tilde X_i$ is a standard normal random variable.   From the well-known equivalence 
$$
\mathcal N(x)\sim \frac{e^{-\frac{x^2}{2}}}{|x|\sqrt{2\pi}},\quad x\to -\infty,
$$
one easily deduces that 
\begin{align}
f_i(x) \sim
\frac{\sigma_i(T_i-T)^{\frac{3}{2}}}{x^2T\sqrt{2\pi}}e^{k_i -
  \frac{d_-^2(x)}{2}},\quad x\to -\infty.  \label{eqf}
\end{align}
Taking the logarithm, we obtain
$$
\ln f_i(x) \sim -\frac{x^2 T}{2(T_i-T)},\quad x\to -\infty
$$
and 
$$
f_i^{-1}(u)\sim \sqrt{2\frac{T_i-T}{T}\ln \frac{1}{u}},\quad u\to
0. 
$$
Therefore, the distribution function of $P_i$ satisfies
$$
\ln \Pr[P_i \leq x] = \ln \mathcal N\{f_i^{-1}(x)\} \sim
-\frac{f_i^{-1}(x)^2}{2} \sim - \frac{T_i-T}{T} \ln \frac{1}{x},\quad
x\downarrow 0,
$$
so that the assumptions of Corollary \ref{main2.thm} are satisfied
with $$\lambda_i = \frac{T_i-T}{T}$$ and $F_0(x) = {1}/{x}$ and the
result follows by applying Proposition \ref{gausstail.prop} and
Corollary \ref{main2.thm}. 
\end{proof}

\paragraph{Numerical illustration}
Figure \ref{optport.fig} plots the distribution function of the
portfolio of three call options written on three different assets, on
the log-log scale. The numerical values of parameters are 
$$
\cm = \left(\begin{aligned}& 0.2 && 0.1 && 0.1 \\ &0.1 && 0.2 && 0.1 \\
    & 0.1 && 0.1 && 0.2\end{aligned}\right), \quad \bs \mu = \left(\begin{aligned}& -0.1 \\ &-0.1 \\
    & -0.1\end{aligned}\right).$$
The time horizon is $T = 0.25$ (years), the option log-strikes are
$k_i = 0$ and the option maturities are $T_i = 0.5$ for $i=1,2,3$. These
values can be considered typical for financial markets. 

The graph plots the distribution function of the option portfolio,
together with the straight line with slope 
$$
\frac{1}{\inf_{\bs w\in \Delta_n}\bs w^\top \Sigma \bs w}
$$
predicted by Proposition \ref{bsopt.prop}, in the log-log scale. We observe power-law decay
in the left tail of the distribution function, and the rate of the
decay (slope of the log-log plot) seems to be close to the theoretical
prediction.  We emphasize the fact that our results may not be used to
actually compute the distribution function, since they only provide
the log-scale asymptotics. Nevertheless, they provide an adequate idea
of the tail behavior of the distribution. 

\begin{figure}
\centerline{\includegraphics[width=0.6\textwidth]{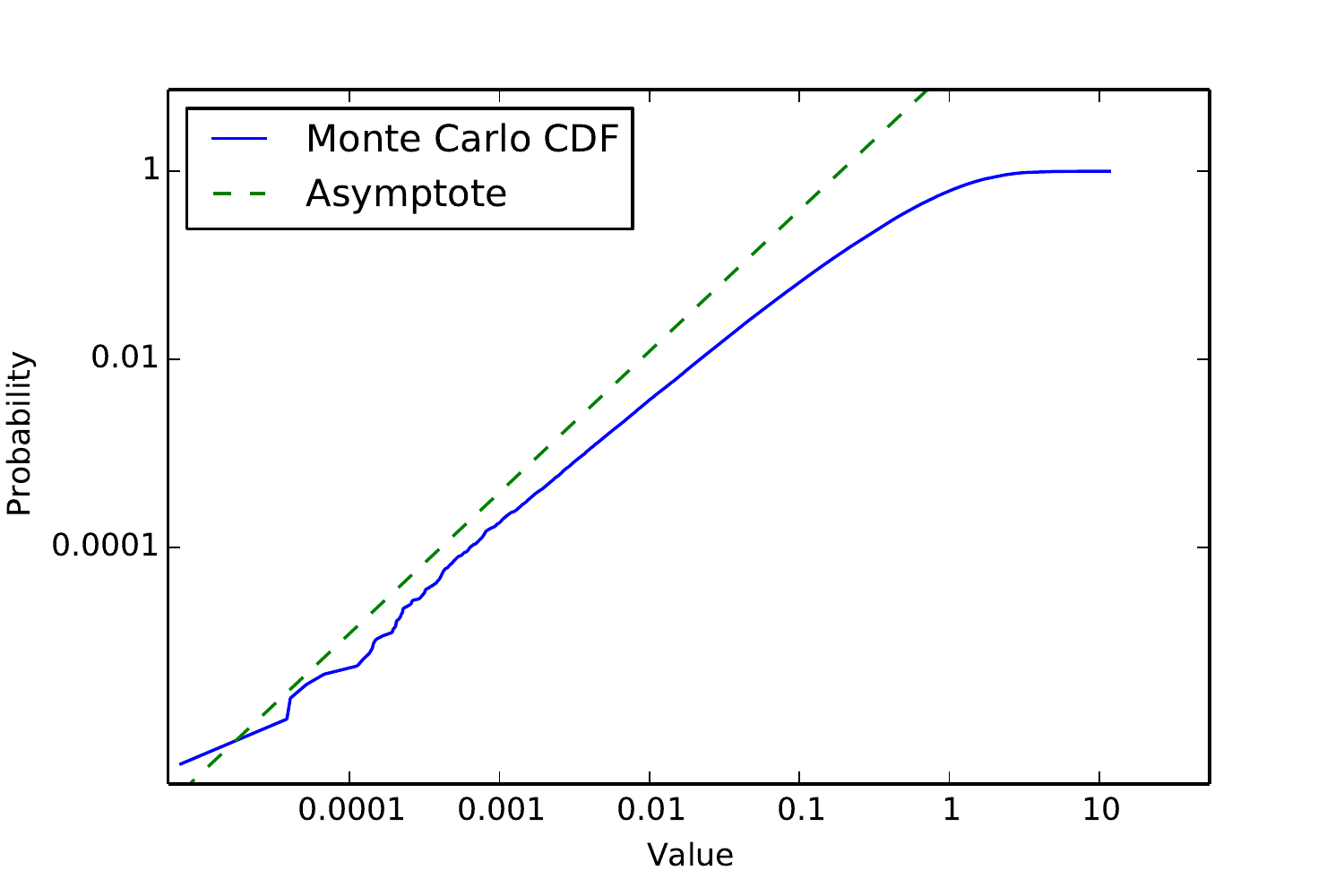}}
\caption{Left tail of the distribution function of the portfolio of
  three call options in a multidimensional Black-Scholes model.}
\label{optport.fig}
\end{figure}

\paragraph{Comparison with hidden regular variation}
When assets and options are identical, the left tail 
behavior of a portfolio of options in the multidimensional
Black-Scholes model can also be analyzed using hidden regular
variation. For the purposes of this illustration, assume that the
portfolio contains two options ($n=2$), leaving the general case for
further research. Let $Z_i = -1/\ln \{1-\mathcal N(\tilde X_i)\} $ for $i=1,2$. It is
known \citep[page 6]{weller2014sum} that for all $z_1,z_2>0$, 
$$
t\Pr\Big[\frac{Z_1}{b_0(t)}> z_1,\frac{Z_2}{b_0(t)}> z_2
\Big]\to (z_1 z_2)^{-\frac{1}{2\eta}}\quad \text{as $t\to
  \infty$}
$$
where $\eta = (1+\rho)/{2}$ and $b_0(t) =  U^{-1}(t)$ with
\begin{align}
U(t) = t^{1/\eta} L_0(t),\quad L_0(t)=  (1+\rho)^{3/2} (1-\rho)^{-1/2} (4\pi \ln
t)^{-\rho/(1+\rho)}. \label{L0}
\end{align}
For all $x_1,x_2>0$, 
$$
\Pr[P_1^{-1}> tx_1, P_2^{-1}(tx_2)] = \Pr[Z_1 > g_1(tx_1),
Z_2 > g_2(tx_2)],
$$
where for $i=1,2$, 
$$
g_i(t) = -\frac{1}{\ln \left[1- \mathcal N\left\{f_i^{-1}\left(\frac{1}{t}\right)\right\}\right]}.
$$
As $t\to \infty$, clearly,
$$
g_i(t) \sim\tilde g_i(t):= \frac{1}{\mathcal N\left\{f_i^{-1}\left(\frac{1}{t}\right)\right\}}.
$$
Moreover, using the equivalent \eqref{eqf}, and the asymptotic
expansion for $\mathcal N^{-1}$ given, e.g., in \cite[page
828]{blair1976rational}, it is easy to show that 
$$
\tilde g_i^{-1}(u) = \frac{1}{f_i\{\mathcal N^{-1}(1/u)\}} \sim C_i u^\frac{T}{T_i - T} (\ln u)^{1-\frac{T}{2(T_i-T)}}
e^{- c_i
  \sqrt{\frac{2T}{T_i-T} \ln u }}
$$
where 
$$
C_i = \frac{2 T \sqrt{2\pi} }{\sigma_i (T_i-T)^{3/2}}
e^{\frac{c_i^2}{2} - k_i} (2\sqrt{\pi})^\frac{T}{T_i-T}\quad\text{and}\quad c_i =
\frac{\mu_i T + k_i}{\sigma_i \sqrt{T_i-T}} + \frac{\sigma_i \sqrt{T_i-T}}{2}.
$$
In other words, $$\tilde g_i^{-1}(u)  = u^{\frac{T}{T_i-T}} L_i(u),$$ where
$L_i$ is a slowly varying function as $u\to \infty$. It follows that
$$\tilde g_i(t) = t^\frac{T_i-T}{T} \tilde L_i(t),$$ where $\tilde L_i$
is slowly varying as $t\to \infty$. By asymptotic inversion we can
show that $\tilde L_i$ satisfies the following relationship.
\begin{align}
\tilde L_i(t) \sim C_i^{-\frac{T_i-T}{T}} \left(\frac{T_i-T}{T}\ln
  t\right)^{\frac{1}{2}-\frac{T_i-T}{T}}
e^{c_i\frac{T_i-T}{T}\sqrt{2\ln t}},\quad t\to \infty. \label{Ltilde}
\end{align}

Now assume that the options and
the assets are identical so that $T_1 = T_2$, $\tilde L_1 = \tilde
L_2 := \tilde L$ and $\tilde g_1 = \tilde g_2:= \tilde g$. Then,
$$
\Pr[P_1^{-1}> tx_1, P_2^{-1}(tx_2)]
\sim\frac{(x_1x_2)^{-\frac{T_1-T}{2\eta T}}}{U\{\tilde g(t)\}}.
$$ 
Therefore, we conclude that the couple $(P_1^{-1},P_2^{-1})$ possesses
the hidden regular variation property, and consequently 
\begin{align*}
\Pr[P_1 + P_2 \leq z] = \Pr[(P_1^{-1}, P_2^{-1}) \in
\frac{1}{z} A ]\sim \frac{1}{U\{\tilde g(1/z)\}} \nu_0(A),
\end{align*}
where 
$$
A = \{x>0, y>0: \frac{1}{x} + \frac{1}{y} \leq 1\}
$$
and $\nu_0$ is a measure defined by 
$$
\nu_0((x_1,\infty)\times (x_2,\infty)) = (x_1x_2)^{-\frac{T_1-T}{2\eta
    T}}.
$$
An easy computation shows that 
$$
\nu_0(A) = \gamma B(1+\gamma, 2+\gamma),
$$
where $B$ is the Euler beta function and $$\gamma =
\frac{T_1-T}{2\eta T}.$$ 

Finally, we have shown that as $z\to 0$, 
$$
\Pr[P_1 + P_2 \leq z] \sim z^{\frac{T_1-T}{\eta T}}\frac{\gamma B(1+\gamma,
  2+\gamma)}{\tilde L(1/z)^{1/\eta} L_0(z^\frac{T-T_1}{T})},
$$
where the function $L_0$ is given explicitly in \eqref{L0} and the
function $\tilde L$ satisfies the asymptotic relation \eqref{Ltilde}. 

It is easy to see that in this case, 
$$
\inf_{\bs w \in \Delta_n} \bs w^\top \Sigma \bs w = \frac{\eta T}{T_1-T},
$$
so that the leading term of the above formula agrees with Proposition
\ref{bsopt.prop}. In conclusion, in this example, the hidden regular
variation theory allows to compute the sharp asymptotics under a
rather restrictive assumption of homogeneous portfolio, while the
methodology of this paper is applicable in the general case but only
enables us to compute the log scale asymptotics. 

\section*{Acknowledgements}
This research was supported by the ANR project FOREWER
  (ANR-14-CE05-0028) and by the chair ``Financial Risks'' of the Risk Foundation, sponsored by Soci\'et\'e G\'en\'erale.


\end{document}